\documentclass[12pt, reqno]{amsart}
\usepackage{amsmath, amsthm, amscd, amsfonts, amssymb, graphicx, color}
\usepackage[bookmarksnumbered, colorlinks, plainpages]{hyperref}
\hypersetup{colorlinks=true,linkcolor=red, anchorcolor=green, citecolor=cyan, urlcolor=red, filecolor=magenta, pdftoolbar=true}

\newtheorem{theorem}{Theorem}[section]
\newtheorem{lemma}[theorem]{Lemma}
\newtheorem{proposition}[theorem]{Proposition}

\theoremstyle{definition}
\newtheorem{definition}[theorem]{Definition}
\newtheorem{example}[theorem]{Example}

\theoremstyle{remark}
\newtheorem{remark}[theorem]{Remark}
\numberwithin{equation}{section}

\begin{document}
	\setcounter{page}{1}
	\title{Controlled $K-g-$fusion frames in Hilbert $C^{\ast}-$modules}
	
	\author{ Fakhr-dine Nhari$^{1}$ and Mohamed Rossafi$^{2*}$}
	
	\address{$^{1}$Laboratory Analysis, Geometry and Applications Department of Mathematics, Faculty Of Sciences, University of Ibn Tofail, Kenitra, Morocco}
	\email{\textcolor[rgb]{0.00,0.00,0.84}{nharidoc@gmail.com}}
	
	\address{$^{2}$LaSMA Laboratory Department of Mathematics, Faculty of Sciences Dhar El Mahraz, University Sidi Mohamed Ben Abdellah, B. P. 1796 Fes Atlas, Morocco}
	\email{\textcolor[rgb]{0.00,0.00,0.84}{rossafimohamed@gmail.com; mohamed.rossafi@usmba.ac.ma}}
	
	
	
	\begin{abstract}
		Controlled frames have been the subject of interest because of its ability to improve the numerical efficiency of iterative algorithms for inverting the frame operator.
		In this paper, we introduce the concepts of controlled $g-$fusion frame and controlled $K-g-$fusion frame in Hilbert  $C^{\ast}-$modules and we give some properties. Also, we study the pertubation problem of controlled $K-g-$fusion frame. Moreover, an illustrative examples is presented to support the obtained results.
	\end{abstract} 
\maketitle
\section{Introduction}
		Frames for Hilbert spaces were introduced by Duffin and Schaefer \cite{Duf} in 1952 to study some deep problems in nonharmonic Fourier series by abstracting the fundamental notion of Gabor \cite{Gab} for signal processing.
		
		Many generalizations of the concept of frame have been defined in Hilbert $C^{\ast}$-modules \cite{Kho2, Lar1, r1, r3, r4, r5, r6, r9}.

		Controlled frames in Hilbert spaces have been introduced by P. Balazs \cite{r10} to improve the numerical efficiency of iterative algorithms for inverting the frame operator.
		
		Rashidi and Rahimi \cite{r11} are introduced the concept of Controlled frames in Hilbert $C^{\ast}-$modules.
		
		The paper is organized in the following manner. In section 3, we introduced the notion of $g-$fusion frames and controlled $g-$fusion frames in Hilbert $C^{\ast}-$modules and estabilish  some  properties. Section 4 is devoted to introduce the concept of controlled $K-g-$fusion frames in Hilbert $C^{\ast}-$modules and gives some results, finally in section 5 we study the perturbation of controlled $K-g-$fusion frames.
		\section{Preliminaires}
		Let $\mathcal{A}$ be a unital $C^{\ast}-$algebra, let $J$ be countable index set. Throughout this paper $H$ and $L$ are countably generated Hilbert $\mathcal{A}-$modules and $\{H_{j}\}_{j\in J}$ is a sequence of submodules of $L$. For each $j\in J$,  $End_{\mathcal{A}}^{\ast}(H,H_{j})$ is the collection of all adjointable $\mathcal{A}-$linear maps from $H$ to $H_{j}$, and $ End_{\mathcal{A}}^{\ast}(H,H)$ is denoted by  $End_{\mathcal{A}}^{\ast}(H)$. Also let $GL^{+}(H)$ be the set of all positive bounded linear invertible operators on $H$ with bounded inverse.
		\begin{definition} \cite{Kap}
			Let $ \mathcal{A} $ be a unital $C^{\ast}$-algebra and $H$ be a left $ \mathcal{A} $-module, such that the linear structures of $\mathcal{A}$ and $H$ are compatible. $H$ is a pre-Hilbert $\mathcal{A}$-module if $H$ is equipped with an $\mathcal{A}$-valued inner product $\langle.,.\rangle :H\times H\rightarrow\mathcal{A}$, such that is sesquilinear, positive definite and respects the module action. In the other words,
			\begin{itemize}
				\item [(i)] $ \langle f,f\rangle\geq0 $ for all $ f\in H $ and $ \langle f,f\rangle=0$ if and only if $f=0$.
				\item [(ii)] $\langle af+g,h\rangle=a\langle f,h\rangle+\langle g,h\rangle$ for all $a\in\mathcal{A}$ and $f,g,h\in H$.
				\item[(iii)] $ \langle f,g\rangle=\langle g,f\rangle^{\ast} $ for all $f,g\in H$.
			\end{itemize}	 
			For $f\in H$, we define $||f||=||\langle f,f\rangle||^{\frac{1}{2}}$. If $H$ is complete with $||.||$, it is called a Hilbert $\mathcal{A}$-module or a Hilbert $C^{\ast}$-module over $\mathcal{A}$. For every $a$ in a $C^{\ast}$-algebra $\mathcal{A}$, we have $|a|=(a^{\ast}a)^{\frac{1}{2}}$ and the $\mathcal{A}$-valued norm on $H$ is defined by $|f|=\langle f, f\rangle^{\frac{1}{2}}$ for $f\in H$.\\
			Define $l^{2}(\{H_{j}\}_{j\in J})$ by
			\begin{equation*}
				l^{2}(\{H_{j}\}_{j\in J})=\{ \{f_{j}\}_{j\in J}:f_{j}\in H_{j},||\sum_{j\in \mathrm{J}}\langle f_{j},f_{j}\rangle||<\infty  \}.
			\end{equation*}
			With $\mathcal{A}-$valued inner product is given by 
			\begin{equation*}
				\langle \{f_{j}\}_{j\in J},\{g_{j}\}_{j\in J}\rangle = \sum_{j\in \mathrm{J}}\langle f_{j},g_{j}\rangle,
			\end{equation*}
			$l^{2}(\{H_{j}\}_{j\in J})$ is a Hilbert $\mathcal{A}-$module. 
		\end{definition}
		The following lemmas was used to proof our results:
		\begin{lemma}\cite{Deh}\label{ALDE}
			If $\phi:\mathcal{A}\rightarrow\mathcal{B}$ is a $\ast -$homomorphism between $C^{\ast}-$algebras, then $\phi$ is increasing, that is, if $a\leq b$, then $\phi(a)\leq\phi(b)$.
		\end{lemma} 
		\begin{lemma}\label{6} \cite{LJI}
			Let $T\in End_{\mathcal{A}}^{\ast}(H,L)$ and $H, L$ are Hilberts $\mathcal{A}-$modules.The following statemnts are multually equivalent:
			\begin{itemize}
				\item [(i)] $T$ is surjective.
				\item [(ii)] $T^{\ast}$ is bounded below with respect to the norm, i.e., there is $m>0$ such that $||T^{\ast}f|| \geq m||f||$ for all $f\in L$.
				\item [(iii)] $T^{\ast}$ is bounded below with respect to the inner product, i.e, there is $m^{'}>0$ such that $\langle T^{\ast}f,T^{\ast}f \rangle \geq m^{'}\langle f,f\rangle $ for  all $f\in L$.
			\end{itemize}
		\end{lemma}
		
		\begin{lemma}\label{3} \cite{Deh} 
			Let $H$ and $L$ are two Hilbert $\mathcal{A}$-modules and $T\in End^{\ast}_{\mathcal{A}}(H,L)$. Then:
			\begin{itemize}
				\item [(i)] If $T$ is injective and $T$ has closed range, then the adjointable map $T^{\ast}T$ is invertible and $$\|(T^{\ast}T)^{-1}\|^{-1}\leq T^{\ast}T\leq\|T\|^{2}.$$
				\item  [(ii)]	If $T$ is surjective, then the adjointable map $TT^{\ast}$ is invertible and $$\|(TT^{\ast})^{-1}\|^{-1}\leq TT^{\ast}\leq\|T\|^{2}.$$
			\end{itemize}
		\end{lemma}
		\begin{lemma}\cite{LJI}\label{11}
			Let $H$ be a Hilbert $\mathcal{A}$-module over a $C^{\ast}$-algebra $\mathcal{A}$, and $T\in End_{\mathcal{A}}^{\ast}(H)$ such that $T^{\ast}=T$. The following statements are equivalent:
			\begin{itemize}
				\item [(i)] $T$ is surjective.
				\item[(ii)] There are $m, M>0$ such that $m\|f\|\leq\|Tf\|\leq M\|f\|$, for all $f\in H$.
				\item[(iii)] There are $m', M'>0$ such that $m'\langle f,f\rangle\leq\langle Tf,Tf\rangle\leq M'\langle f,f\rangle$ for all $f\in H$.
			\end{itemize}
		\end{lemma}
		\begin{lemma} \cite{Pas}
			Let $H$ be a Hilbert $\mathcal{A}$-module. If $T\in End_{\mathcal{A}}^{\ast}(H)$, then 
			\begin{equation*}
				\langle Tf,Tf\rangle\leq\|T\|^{2}\langle f,f\rangle,\qquad \forall f\in H.
			\end{equation*}
		\end{lemma}
		\begin{lemma}\cite{Fang}\label{5}
		Let $E,H$ and $L$ be Hilbert $\mathcal{A}-$modules, $T\in End_{\mathcal{A}}^{\ast}(E,L)$ and $T^{'}\in End_{\mathcal{A}}^{\ast}(H,L)$. Then the following two statements are equivalent:
		\begin{itemize}
			\item[(1)] $T^{'}(T^{'})^{\ast}\leq \lambda TT^{\ast}$ for some $\lambda>0$;
			\item[(2)] There exists $\mu>0$ such that $\|(T^{'})^{\ast}z\|\leq \mu\|T^{\ast}z\|$ for all $z\in L$.
		\end{itemize}
		\end{lemma}
		
		\begin{lemma}\cite{rfn1}\label{9}
			Let $\{W_{j}\}_{j\in J}$ be a sequence of orthogonally complemented closed submodules of $H$ and $T\in End_{\mathcal{A}}^{\ast}(H)$ invertible, if $T^{\ast}TW_{j}\subset  W_{j}$ for each $j \in J$, then $\{TW_{j}\}_{j\in J}$ is a sequence of orthogonally complemented closed submodules and $P_{W_{j}}T^{\ast}=P_{W_{j}}T^{\ast}P_{TW_{j}}$.
		\end{lemma}
		\section{Controlled $g-$fusion frame in Hilbert $C^{\ast}-$modules}
		
		Firstly we give the definition of $g-$fusion frame in Hilbert $C^{\ast}-$modules.
		\begin{definition}\cite{rfn1}
			Let $\{W_{j}\}_{j\in J}$ be a sequence of closed submodules orthogonally complemented of $H$, $\{v_{j}\}_{j\in J}$ be a family of weights in $\mathcal{A}$, ie., each $v_{j}$ is positive invertible element frome the center of $\mathcal{A}$ and $\Lambda_{j}\in End_{\mathcal{A}}^{\ast}(H,H_{j})$ for each $j\in J$.  We say that $\Lambda=\{W_{j},\Lambda_{j},v_{j}\}_{j\in J}$ is a $g-$fusion frame for $H$ if there exists $0<A\leq B<\infty$ such that 
			\begin{equation}
				A\langle f,f \rangle \leq \sum_{j\in J}v_{j}^{2}\langle \Lambda_{j}P_{W_{j}}f,\Lambda_{j}P_{W_{j}}f\rangle\leq B\langle f,f\rangle,\qquad \forall f\in H.
			\end{equation}
			The constants  $A$ and $B$ are called the lower and upper bounds of the  $g-$fusion frame, respectively. If $A=B$ then $\Lambda$ is called tight $g-$fusion frame and if $A=B=1$ then we say $\Lambda$ is a Parseval $g-$fusion frame.
			
			The operator $S:H\rightarrow H$ defined by
			\begin{equation*}
				Sf=\sum_{j\in J}v_{j}^{2}P_{W_{j}}\Lambda_{j}^{\ast}\Lambda_{j}P_{W_{j}}f,\qquad \forall f\in H.
			\end{equation*}
			Is called $g-$fusion frame operator. 
		\end{definition}
		Now we define the notion of $(C,C^{'})-$controlled $g-$fusion frame in Hilbert $C^{\ast}-$modules.
		\begin{definition}
			Let $C$, $C^{'}\in GL^{+}(H)$, $\{W_{j}\}_{j\in J}$ be a sequence of closed submodules orthogonally complemented of $H$, $\{v_{j}\}_{j\in J}$ be a family of weights in $\mathcal{A}$, i.e., each $v_{j}$ is a positive invertible element frome the center of $\mathcal{A}$ and $\Lambda_{j}\in End_{\mathcal{A}}^{\ast}(H,H_{j})$ for each $j\in J$. We say that $\Lambda_{CC^{'}}=\{W_{j},\Lambda_{j},v_{j}\}_{j\in J}$ is a $(C,C^{'})-$controlled $g-$fusion frame for $H$ if there exists $0<A\leq B<\infty$ such that 
			\begin{equation}\label{eq1}
				A\langle f,f \rangle \leq \sum_{j\in J}v_{j}^{2}\langle \Lambda_{j}P_{W_{j}}Cf,\Lambda_{j}P_{W_{j}}C^{'}f\rangle\leq B\langle f,f\rangle,\qquad \forall f\in H.
			\end{equation}
			The constants $A$ and $B$ are called the lower and upper bounds of the  $(C,C^{'})-$controlled $g-$fusion frame, respectively.
			When $A=B$, the sequence $\Lambda_{CC^{'}}=\{W_{j},\Lambda_{j},v_{j}\}_{j\in J}$ is called  $(C,C^{'})-$controlled tight $g-$fusion frame, and when $A=B=1$, it is called a $(C,C^{'})-$controlled Parseval $g-$fusion frame. If only upper inequality of \eqref{eq1} hold, then $\Lambda_{CC^{'}}$ is called an $(C,C^{'})-$controlled $g-$fusion bessel sequence for $H$. 
		\end{definition} 
		\begin{example}
			Let $l^{\infty}$ be the set of all bounded complex-valued sequences. For any $u=\{u_{j}\}_{j\in \mathbb{N}}$, $v=\{v_{j}\}_{j\in \mathbb{N}}\in l^{\infty}$, we have 
			\begin{equation*}
				uv=\{u_{j}v_{j}\}_{j\in \mathbb{N}}, u^{\ast}=\{\overline{u_{j}}\}_{j\in \mathbb{N}}, ||u||=\sup_{j\in\mathbb{N}}|u_{j}|.
			\end{equation*}
			Then $\mathcal{A}=\{l^{\infty}, ||.||\}$ is a $C^{\ast}-$algebra. 
			
			Let $H=C_{0}$ be the set of all sequences converging to zero. For any $u$, $v\in H$ we define 
			\begin{equation*}
				\langle u,v\rangle=uv^{\ast}=\{u_{j}\overline{v_{j}}\}_{j\in\mathbb{N}}.
			\end{equation*}
			Then $H$ is a Hilbert $\mathcal{A}-$module.
			
			Now let $\{e_{j}\}_{j\in\mathbb{N}}$ be the standard orthonormal basis of $H$. 
			
			We construct $H_{j}=\overline{span}\{e_{1},e_{2},...,e_{j}\}$ and $W_{j}=\overline{span}\{e_{j}\}$ for each $j\in\mathbb{N}$.\\
			Define $\Lambda_{j}:H\rightarrow H_{j}$ by $\Lambda_{j}(f)=\sum_{k=1}^{j}\langle f,\frac{e_{j}}{\sqrt{j}}\rangle e_{k}$.\\
			The adjoint operator $\Lambda_{j}^{\ast}:H_{j}\rightarrow H$ define by $\Lambda_{j}^{\ast}(g)=\sum_{k=1}^{j}\langle g,\frac{e_{k}}{\sqrt{j}}\rangle e_{j}$.\\
			And the projection orthogonal $P_{W_{j}}$ define by $P_{W_{j}}(f)=\langle f,e_{j}\rangle e_{j}$.\\
			Let us define $Cf=2f$ and $C^{'}f=\frac{1}{2}f$. Then for any $f\in H$, we have 
			\begin{align*}
				\langle \Lambda_{j}P_{W_{j}}Cf,\Lambda_{j}P_{W_{j}}C^{'}f\rangle&=\langle \frac{2}{\sqrt{j}}\langle f,e_{j}\rangle\sum_{k=1}^{j}e_{k},\frac{1}{2\sqrt{j}}\langle f,e_{j}\rangle\sum_{k=1}^{j}e_{k}\rangle\\&=\frac{1}{j}\langle f, e_{j}\rangle\langle e_{j},f\rangle\langle \sum_{k=1}^{j}e_{k},\sum_{k=1}^{j}e_{k}\rangle\\&=\frac{1}{j}\langle f, e_{j}\rangle\langle e_{j},f\rangle\sum_{k=1}^{j}||e_{k}||^{2}\\&=\frac{1}{j}\langle f, e_{j}\rangle\langle e_{j},f\rangle j\\&=\langle f, e_{j}\rangle\langle e_{j},f\rangle.
			\end{align*}
			Therefore, for each $f\in H$, 
			\begin{equation*}
				\sum_{j\in\mathbb{N}}\langle \Lambda_{j}P_{W_{j}}Cf,\Lambda_{j}P_{W_{j}}C^{'}f\rangle=\sum_{j\in\mathbb{N}}\langle f, e_{j}\rangle\langle e_{j},f\rangle=\langle f,f\rangle.
			\end{equation*}
			Hence $\{W_{j},\Lambda_{j},1\}_{j\in\mathbb{N}}$ is a $(C,C^{'})-$controlled Parseval $g-$fusion frame for $H$.
		\end{example}
		Suppose that $\Lambda_{CC^{'}}$ be a $(C,C^{'})-$controlled $g-$fusion bessel sequence for $H$. The bounded linear operator $T_{(C,C^{'})}:l^{2}(\{H_{j}\}_{j\in J})\rightarrow H$ define by 
		\begin{equation}
			T_{(C,C^{'})}(\{f_{j}\}_{j\in J})=\sum_{j\in J}v_{j}(CC^{'})^{\frac{1}{2}}P_{W_{j}}\Lambda_{j}^{\ast}f_{j},\qquad \forall \{f_{j}\}_{j\in J}\in l^{2}(\{H_{j}\}_{j\in J}).
		\end{equation}
		is called the synthesis operator for the $(C,C^{'})-$controlled $g-$fusion frame $\Lambda_{CC^{'}}$.\\
		The adjoint operator $T_{(C,C^{'})}^{\ast}:H\rightarrow l^{2}(\{H_{j}\}_{j\in J})$ given by 
		\begin{equation}\label{eq4}
			T_{(C,C^{'})}^{\ast}(g)=\{v_{j}\Lambda_{j}P_{W_{j}}(C^{'}C)^{\frac{1}{2}}g\}_{j\in J}
		\end{equation}
		is called the analysis operator for the $(C,C^{'})-$controlled $g-$fusion frame $\Lambda_{CC^{'}}$.\\
		When  $C$ and $C^{'}$ commute with each other, and commute with the operator $P_{W_{j}}\Lambda_{j}^{\ast}\Lambda_{j}P_{W_{j}}$, for each $j\in J$, then the $(C,C^{'})-$controlled $g-$fusion frame operator $S_{(C,C^{'})}:H\rightarrow H$ is defined as 
		\begin{equation}
			S_{(C,C^{'})}(f)=T_{(C,C^{'})}T_{(C,C^{'})}^{\ast}(f)=\sum_{j\in J}v_{j}^{2}C^{'}P_{W_{j}}\Lambda_{j}^{\ast}\Lambda_{j}P_{W_{j}}Cf,\qquad \forall f\in H.
		\end{equation}
		And we have 
		\begin{equation}\label{eq12}
			\langle S_{(C,C^{'})}(f),f\rangle =\sum_{j\in J}v_{j}^{2}\langle\Lambda_{j}P_{W_{j}}Cf,\Lambda_{j}P_{W_{j}}C^{'}f\rangle,\qquad\forall f\in H.
		\end{equation}
		From now we assume that  $C$ and $C^{'}$ commute with each other, and commute with the operator $P_{W_{j}}\Lambda_{j}^{\ast}\Lambda_{j}P_{W_{j}}$, for each $j\in J$
		\begin{lemma}
			Let $\Lambda_{CC^{'}}$ be a $(C,C^{'})-$controlled $g-$fusion frame for $H$. Then the  $(C,C^{'})-$controlled $g-$fusion frame operator $S_{(C,C^{'})}$ is positive, self-adjoint and invertible.
		\end{lemma}
		\begin{proof}
			For each $f\in H$ we have  
			\begin{equation*}
				S_{(C,C^{'})}(f)=\sum_{j\in J}v_{j}^{2}C^{'}P_{W_{j}}\Lambda_{j}^{\ast}\Lambda_{j}P_{W_{j}}Cf
			\end{equation*}
			Then
			\begin{equation*}
				\sum_{j\in J}v_{j}^{2}\langle\Lambda_{j}P_{W_{j}}Cf,\Lambda_{j}P_{W_{j}}C^{'}f\rangle=\langle\sum_{j\in J}v_{j}^{2}C^{'}P_{W_{j}}\Lambda_{j}^{\ast}\Lambda_{j}P_{W_{j}}Cf,f\rangle=\langle S_{(C,C^{'})}(f),f\rangle.
			\end{equation*}   
			Since $\Lambda_{CC^{'}}$ is a $(C,C^{'})-$controlled $g-$fusion frame for $H$, then 
			\begin{equation}\label{eq2}
				A\langle f,f \rangle \leq \langle S_{(C,C^{'})}(f),f\rangle \leq B\langle f,f\rangle,\qquad \forall f\in H
			\end{equation}
			It is clear that $S_{(C,C^{'})}$ is  positive, bounded and linear operator. On the other hand for each $f$, $g\in H$
			\begin{align*}
				\langle S_{(C,C^{'})}(f),g\rangle&=\langle\sum_{j\in J}v_{j}^{2}C^{'}P_{W_{j}}\Lambda_{j}^{\ast}\Lambda_{j}P_{W_{j}}Cf,g\rangle\\&=\langle f,\sum_{j\in J}v_{j}^{2}CP_{W_{j}}\Lambda_{j}^{\ast}\Lambda_{j}P_{W_{j}}C^{'}g\rangle\\&=\langle f,S_{(C^{'},C)}(g)\rangle.
			\end{align*}
			That implies $S_{(C,C^{'})}^{\ast}=S_{(C^{'},C)}$. Also as $C$ and $C^{'}$ commute with each other, and commute with the operator $P_{W_{j}}\Lambda_{j}^{\ast}\Lambda_{j}P_{W_{j}}$, for each $j\in J$, we have $S_{(C,C^{'})}=S_{(C^{'},C)}$. So the $(C,C^{'})-$controlled $g-$fusion frame operator $S_{(C,C^{'})}$ is self-adjoint. And from inequality \eqref{eq2} we have
			\begin{equation}
				AI_{H}\leq S_{(C,C^{'})}\leq BI_{H}.
			\end{equation} 	
			Therefore, the $(C,C^{'})-$controlled $g-$fusion frame operator $S_{(C,C^{'})}$ is invertible.
		\end{proof}
		We estabilish an equivalent definition of $(C,C^{'})-$controlled $g-$fusion frame.
		\begin{theorem}\label{2}
			$\Lambda_{CC^{'}}=\{W_{j},\Lambda_{j},v_{j}\}_{j\in J}$ is a $(C,C^{'})-$controlled $g-$fusion frame for $H$. If and only if there exists two constants $0<A\leq B<\infty$ such that 
			\begin{equation}\label{eq3}
				A||f||^{2}\leq ||\sum_{j\in J}v_{j}^{2}\langle \Lambda_{j}P_{W_{j}}Cf,\Lambda_{j}P_{W_{j}}C^{'}f\rangle ||\leq B||f||^{2},\qquad \forall f\in H. 
			\end{equation}
		\end{theorem}
		\begin{proof}
			If $\Lambda_{CC^{'}}$ be a $(C,C^{'})-$controlled $g-$fusion frame for $H$, then we have inequality \eqref{eq3}.\\
			Converselly, assume that \eqref{eq3} holds. From \eqref{eq4}, the $(C,C^{'})-$controlled $g-$fusion frame operator $S_{(C,C^{'})}$ is positive, self-adjoint and invertible. Then we have for all $f\in H$
			\begin{equation}\label{eq5}
				\langle (S_{(C,C^{'})})^{\frac{1}{2}}f,(S_{(C,C^{'})})^{\frac{1}{2}}f\rangle=\langle S_{(C,C^{'})}f,f\rangle=\sum_{j\in J}v_{j}^{2}\langle \Lambda_{j}P_{W_{j}}Cf,\Lambda_{j}P_{W_{j}}C^{'}f\rangle.
			\end{equation}
			Using \eqref{eq3} and \eqref{eq5}, we conclude that
			\begin{equation*}
				\sqrt{A}||f||\leq ||S_{(C,C^{'})}^{\frac{1}{2}}f||\leq \sqrt{B}||f||,\qquad\forall f\in H.
			\end{equation*}
			So by lemma \ref{11}, $\Lambda_{CC^{'}}$ is a $(C,C^{'})-$controlled $g-$fusion frame for $H$. 
		\end{proof}
		\begin{theorem}
			Let $\{W_{j},\Lambda_{j},v_{j}\}_{j\in J}$ be a $g-$fusion frame for $H$ with frame operator $S$ and let $C$, $C^{'}\in GL^{+}(H)$. Then $\{W_{j},\Lambda_{j},v_{j}\}_{j\in J}$ is a  $(C,C^{'})-$controlled $g-$fusion frame for $H$.
		\end{theorem}
		\begin{proof}
			Let  $\{W_{j},\Lambda_{j},v_{j}\}_{j\in J}$ be a $g-$fusion frame for $H$ with frame bounds $A$ and $B$. Then for each $f\in H$ 
			\begin{equation}\label{eq15}
				A\langle f,f \rangle \leq \sum_{j\in J}v_{j}^{2}\langle \Lambda_{j}P_{W_{j}}f,\Lambda_{j}P_{W_{j}}f\rangle\leq B\langle f,f\rangle
			\end{equation}
			We have 
			\begin{equation}\label{eq16}
				||\sum_{j\in J}v_{j}^{2}\langle \Lambda_{j}P_{W_{j}}Cf,\Lambda_{j}P_{W_{j}}C^{'}f\rangle||=||\langle S_{(C,C^{'})}f,f\rangle ||=||C||.||C^{'}||.||\langle Sf,f\rangle ||,
			\end{equation}
			Using \eqref{eq15} and \eqref{eq16}, we conclude
			\begin{equation*}
				A||C||.||C^{'}||||\langle f,f\rangle ||\leq ||\sum_{j\in J}v_{j}^{2}\langle \Lambda_{j}P_{W_{j}}Cf,\Lambda_{j}P_{W_{j}}C^{'}f\rangle||\leq B||C||.||C^{'}||||\langle f,f\rangle||,\qquad \forall f\in H.
			\end{equation*}
			Therefore,  $\{W_{j},\Lambda_{j},v_{j}\}_{j\in J}$ is a  $(C,C^{'})-$controlled $g-$fusion frame for $H$ with bounds $A||C||.||C^{'}||$ and $B||C||.||C^{'}||$.
		\end{proof}
		\begin{remark}
			When $C=C^{'}$ we say that the sequence $\{W_{j},\Lambda_{j},v_{j}\}_{j\in J}$ is a $C^{2}-$controlled $g-$fusion frame for $H$.
		\end{remark}
		\begin{theorem}
			Let $C\in GL^{+}(H)$. The sequence $\{W_{j},\Lambda_{j},v_{j}\}_{j\in J}$ is a $g-$fusion frame for $H$ if and only if $\{W_{j},\Lambda_{j},v_{j}\}_{j\in J}$ is a $C^{2}-$controlled $g-$fusion frame for $H$. 
		\end{theorem}
		\begin{proof}
			Suppose that $\{W_{j},\Lambda_{j},v_{j}\}_{j\in J}$ is a $g-$fusion frame for $H$. with bounds $A$ and $B$. Then 
			\begin{equation*}
				A\langle f,f\rangle \leq\sum_{j\in J}v_{j}^{2}\langle \Lambda_{j}P_{W_{j}}f,\Lambda_{j}P_{W_{j}}f\rangle \leq B\langle f,f\rangle,\qquad \forall f\in H.
			\end{equation*}
			We have for each $f\in H$,
			\begin{equation}\label{eq6}
				\sum_{j\in J}v_{j}^{2}\langle \Lambda_{j}P_{W_{j}}Cf,\Lambda_{j}P_{W_{j}}Cf\rangle\leq B\langle Cf,Cf\rangle\leq B||C||^{2}\langle f,f\rangle.
			\end{equation}
			On the other hand for each $f\in H$
			\begin{align}\label{eq7}
				A\langle f,f\rangle=A\langle C^{-1}Cf,C^{-1}Cf\rangle &\leq A||C^{-1}||^{2}\langle Cf,Cf\rangle\notag \\&\leq ||C^{-1}||^{2}\sum_{j\in J}v_{j}^{2}\langle \Lambda_{j}P_{W_{j}}Cf,\Lambda_{j}P_{W_{j}}Cf\rangle.
			\end{align}
			So from \eqref{eq6} and \eqref{eq7}, we have  
			\begin{equation*}
				A||C^{-1}||^{-2}\langle f,f\rangle \leq \sum_{j\in J}v_{j}^{2}\langle \Lambda_{j}P_{W_{j}}Cf,\Lambda_{j}P_{W_{j}}Cf\rangle\leq B||C||^{2}\langle f,f\rangle,\qquad \forall f\in H.
			\end{equation*}
			We conclude that $\{W_{j},\Lambda_{j},v_{j}\}_{j\in J}$ is a $C^{2}-$controlled $g-$fusion frame for $H$.\\
			Converselly, Let $\{W_{j},\Lambda_{j},v_{j}\}_{j\in J}$ be a $C^{2}-$controlled $g-$fusion frame for $H$ with bounds $A^{'}$ and $B^{'}$. Then for all $f\in H$,
			\begin{equation*}
				A^{'}\langle f,f\rangle\leq \sum_{j\in J}v_{j}^{2}\langle \Lambda_{j}P_{W_{j}}Cf,\Lambda_{j}P_{W_{j}}Cf\rangle\leq B^{'}\langle f,f\rangle
			\end{equation*}
			We have for each $f\in H$,
			\begin{align}\label{eq8}
				\sum_{j\in J}v_{j}^{2}\langle \Lambda_{j}P_{W_{j}}f,\Lambda_{j}P_{W_{j}}f\rangle &\notag=\sum_{j\in J}v_{j}^{2}\langle \Lambda_{j}P_{W_{j}}CC^{-1}f,\Lambda_{j}P_{W_{j}}CC^{-1}f\rangle \\&\notag\leq B^{'}\langle C^{-1}f,C^{-1}f\rangle  \\&\leq B^{'}||C^{-1}||^{2}\langle f,f\rangle.
			\end{align}
			Also for each $f\in H$,
			\begin{align*}
				A^{'}\langle C^{-1}f,C^{-1}f\rangle\leq \sum_{j\in J}v_{j}^{2}\langle \Lambda_{j}P_{W_{j}}CC^{-1}f,\Lambda_{j}P_{W_{j}}CC^{-1}f\rangle=\sum_{j\in J}v_{j}^{2}\langle \Lambda_{j}P_{W_{j}}f,\Lambda_{j}P_{W_{j}}f\rangle
			\end{align*}
			And
			\begin{equation}\label{eq9}
				A^{'}||(C^{-1}C^{-1})^{-1}||^{-1}\langle f,f\rangle \leq 	A^{'}\langle C^{-1}f,C^{-1}f\rangle\leq \sum_{j\in J}v_{j}^{2}\langle \Lambda_{j}P_{W_{j}}f,\Lambda_{j}P_{W_{j}}f\rangle
			\end{equation}
			From \eqref{eq8} and \eqref{eq9}, we have
			\begin{equation*}
				A^{'}||(C^{-2})^{-1}||^{-1}\langle f,f\rangle \leq \sum_{j\in J}v_{j}^{2}\langle \Lambda_{j}P_{W_{j}}f,\Lambda_{j}P_{W_{j}}f\rangle\leq B^{'}||C^{-1}||^{2}\langle f,f\rangle,\qquad \forall f\in H.
			\end{equation*}
			Hence $\{W_{j},\Lambda_{j},v_{j}\}_{j\in J}$ is a $g-$fusion frame for $H$.
		\end{proof}
		\begin{theorem}
			Let $C$, $C^{'}\in GL^{+}(H)$, and $C$, $C^{'}$ commute with each other and commute with $P_{W_{j}}\Lambda_{j}^{\ast}\Lambda_{j}P_{W_{j}}$ for all $j\in J$. Then $\Lambda_{CC^{'}}=\{W_{j},\Lambda_{j},v_{j}\}_{j\in J}$ is a $(C,C^{'})-$controlled $g-$fusion bessel sequence for $H$ with bound $B$ if and only if the operator $T_{(C,C^{'})}:l^{2}(\{H_{j}\}_{j\in J}) \rightarrow H$ given by 
			\begin{equation*}
				T_{(C,C^{'})}(\{g_{j}\}_{j\in J})=\sum_{j\in J}v_{j}(CC^{'})^{\frac{1}{2}}P_{W_{j}}\Lambda_{j}^{\ast}g_{j},\qquad \forall \{g_{j}\}_{j\in J}\in l^{2}(\{H_{j}\}_{j\in J}). 
			\end{equation*}
			is well defined and bounded operator with, $||T_{(C,C^{'})}||\leq \sqrt{B}$.
		\end{theorem}
		\begin{proof}
			Let $\Lambda_{CC^{'}}$ is a $(C,C^{'})-$controlled $g-$fusion bessel sequence with bound $B$ for $H$. As a result of theorem \ref{2},
			\begin{equation}
				|| \sum_{j\in J}v_{j}^{2}\langle \Lambda_{j}P_{W_{j}}Cf,\Lambda_{j}P_{W_{j}}C^{'}f\rangle||\leq B||f||^{2},\qquad \forall f\in H.
			\end{equation}
			For any $\{g_{j}\}_{j\in J}\in l^{2}(\{H_{j}\}_{j\in J})$,
			\begin{align*}
				||T_{(C,C^{'})}(\{g_{j}\}_{j\in J})||&=\sup_{||f||=1}||\langle T_{(C,C^{'})}(\{g_{j}\}_{j\in J}),f\rangle||\\&=\sup_{||f||=1}||\langle \sum_{j\in J}v_{j}(CC^{'})^{\frac{1}{2}}P_{W_{j}}\Lambda_{j}^{\ast}g_{j},f\rangle||\\&=\sup_{||f||=1}||\sum_{j\in J}\langle v_{j}(CC^{'})^{\frac{1}{2}}P_{W_{j}}\Lambda_{j}^{\ast}g_{j},f\rangle||\\&=\sup_{||f||=1}||\sum_{j\in J}\langle g_{j}, v_{j}\Lambda_{j}P_{W_{j}}(CC^{'})^{\frac{1}{2}}f\rangle||\\&\leq \sup_{||f||=1}||\sum_{j\in J}\langle g_{j}, g_{j}\rangle||^{\frac{1}{2}}||\sum_{j\in J}v_{j}^{2}\langle\Lambda_{j}P_{W_{j}}(CC^{'})^{\frac{1}{2}}f, \Lambda_{j}P_{W_{j}}(CC^{'})^{\frac{1}{2}}f\rangle||^{\frac{1}{2}}\\&= \sup_{||f||=1}||\sum_{j\in J}\langle g_{j}, g_{j}\rangle||^{\frac{1}{2}}||\sum_{j\in J}v_{j}^{2}\langle\Lambda_{j}P_{W_{j}}Cf, \Lambda_{j}P_{W_{j}}C^{'}f\rangle||^{\frac{1}{2}}\\&\leq \sup_{||f||=1}||\sum_{j\in J}\langle g_{j}, g_{j}\rangle||^{\frac{1}{2}}\sqrt{B}||f||=\sqrt{B}||\{g_{j}\}_{j\in J}||.
			\end{align*}
			Therefore, the sum $\sum_{j\in J}v_{j}(CC^{'})^{\frac{1}{2}}P_{W_{j}}\Lambda_{j}^{\ast}g_{j}$ is convergent, and we have 
			\begin{equation*}
				||T_{(C,C^{'})}(\{g_{j}\}_{j\in J})||\leq \sqrt{B}||\{g_{j}\}_{j\in J}||
			\end{equation*}
			Hence the operator $T_{(C,C^{'})}$ is well defined, bounded and $||T_{(C,C^{'})}||\leq \sqrt{B}$.\\
			For the converse, suppose that the operator $T_{(C,C^{'})}$ is well defined, bounded and $||T_{(C,C^{'})}||\leq \sqrt{B}$. For all $f\in H$, we have 
			\begin{align*}
				||\sum_{j\in J}v_{j}^{2}\langle \Lambda_{j}P_{W_{j}}Cf,\Lambda_{j}P_{W_{j}}C^{'}f\rangle|| &=||\sum_{j\in J}v_{j}^{2}\langle C^{'}P_{W_{j}}\Lambda_{j}^{\ast}\Lambda_{j}P_{W_{j}}Cf,f\rangle||\\&=||\sum_{j\in J}v_{j}^{2}\langle (CC^{'})^{\frac{1}{2}}P_{W_{j}}\Lambda_{j}^{\ast}\Lambda_{j}P_{W_{j}}(CC^{'})^{\frac{1}{2}}f,f\rangle||\\&=||\langle T_{(C,C^{'})}(\{g_{j}\}_{j\in J}),f\rangle||\\&\leq ||T_{(C,C^{'})}||||\{g_{j}\}_{j\in J}||||f||\\&=||T_{(C,C^{'})}||||\sum_{j\in J}v_{j}^{2}\langle \Lambda_{j}P_{W_{j}}Cf,\Lambda_{j}P_{W_{j}}C^{'}f\rangle||^{\frac{1}{2}}||f||
			\end{align*}
			Where $g_{j}=v_{j}\Lambda_{j}P_{W_{j}}(CC^{'})^{\frac{1}{2}}f$.\\
			Hence 
			\begin{equation*}
				||\sum_{j\in J}v_{j}^{2}\langle \Lambda_{j}P_{W_{j}}Cf,\Lambda_{j}P_{W_{j}}C^{'}f\rangle||^{\frac{1}{2}}\leq \sqrt{B}||f||
			\end{equation*}
			Then
			\begin{equation}\label{eq25}
				||\sum_{j\in J}v_{j}^{2}\langle \Lambda_{j}P_{W_{j}}Cf,\Lambda_{j}P_{W_{j}}C^{'}f\rangle||\leq B||f||^{2}
			\end{equation}
			The adjoint operator of $T_{(C,C^{'})}$ is given by 
			\begin{equation*}
				T_{(C,C^{'})}^{\ast}(g)=\{v_{j}\Lambda_{j}P_{W_{j}}(CC^{'})^{\frac{1}{2}}g\}_{j\in J},\qquad\forall g\in H.
			\end{equation*}
			And we have for each $f\in H$
			\begin{align*}
				||\sum_{j\in J}\langle\Lambda_{j}P_{W_{j}}Cf,\Lambda_{j}P_{W_{j}}C^{'}f\rangle||&=||\sum_{j\in J}v_{j}^{2}\langle \Lambda_{j}P_{W_{j}}(CC^{'})^{\frac{1}{2}}f,\Lambda_{j}P_{W_{j}}(CC^{'})^{\frac{1}{2}}f\rangle||\\&=||\langle T^{\ast}_{(C,C^{'})}(f),T^{\ast}_{(C,C^{'})}(f)\rangle||\\&=||T^{\ast}_{(C,C^{'})}(f)||^{2}
			\end{align*}
			Frome \eqref{eq25}, we have 
			\begin{equation*}
				||T^{\ast}_{(C,C^{'})}(f)||\leq \sqrt{B}||f||,\qquad\forall f\in H.
			\end{equation*}
			So, $T^{\ast}_{(C,C^{'})}$ is bounded $\mathcal{A}-$linear operator, then there exist a constant $M>0$ such that 
			\begin{equation*}
				\langle T^{\ast}_{(C,C^{'})}f,T^{\ast}_{(C,C^{'})}f\rangle \leq M\langle f,f\rangle,\qquad\forall f\in H.
			\end{equation*} 
			Hence 
			\begin{equation*}
				\sum_{j\in J}v_{j}^{2}\langle \Lambda_{j}P_{W_{j}}Cf,\Lambda_{j}P_{W_{j}}C^{'}f\rangle\leq M\langle f,f\rangle.\qquad\forall f\in H.
			\end{equation*}
			This give that $\Lambda_{CC^{'}}$ is a $(C,C^{'})-$controlled $g-$fusion bessel sequence for $H$.
		\end{proof}
		\begin{theorem}
			Let $\{W_{j},\Lambda_{j},v_{j}\}_{j\in J}$ be a $(C,C^{'})-$controlled $g-$fusion frame for $H$ with bounds $A$ and $B$, with operator frame $S_{(C,C^{'})}$. Let $\theta \in End_{\mathcal{A}}^{\ast}(H)$ be injective and has a closed range. Suppose that $\theta$ commute with $C$, $C^{'}$ and $P_{W_{j}}$ for all $j\in J$. Then $\{W_{j},\Lambda_{j}\theta,v_{j}\}_{j\in J}$ is a $(C,C^{'})-$controlled $g-$fusion frame for $H$.
		\end{theorem}
		\begin{proof}
			Let $\{W_{j},\Lambda_{j},v_{j}\}_{j\in J}$ be a $(C,C^{'})-$controlled $g-$fusion frame for $H$ with bounds $A$ and $B$, then 
			\begin{equation*}
				A\langle f,f \rangle \leq \sum_{j\in J}v_{j}^{2}\langle \Lambda_{j}P_{W_{j}}Cf,\Lambda_{j}P_{W_{j}}C^{'}f\rangle\leq B\langle f,f\rangle,\qquad \forall f\in H.
			\end{equation*}
			For each $f\in H$, we have 
			\begin{align}\label{eq21}
				\sum_{j\in J}v_{j}^{2}\langle \Lambda_{j}\theta P_{W_{j}}Cf,\Lambda_{j}\theta P_{W_{j}}C^{'}f\rangle &= \sum_{j\in J}v_{j}^{2}\langle \Lambda_{j}P_{W_{j}}C\theta f,\Lambda_{j}P_{W_{j}}C^{'}\theta f\rangle\notag\\&\leq B\langle \theta f,\theta f\rangle\notag \\&\leq B||\theta ||^{2}\langle f,f \rangle
			\end{align}
			And 
			\begin{equation*}
				A\langle \theta f,\theta f\rangle \leq \sum_{j\in J}v_{j}^{2}\langle \Lambda_{j}\theta P_{W_{j}}Cf,\Lambda_{j}\theta P_{W_{j}}C^{'}f\rangle,
			\end{equation*}
			By lemma \ref{3}, we have 
			\begin{equation*}
				A||(\theta^{\ast}\theta)^{-1}||^{-1}\langle f,f\rangle \leq A\langle \theta f,\theta f\rangle
			\end{equation*}
			So
			\begin{equation}\label{eq18}
				A||(\theta^{\ast}\theta)^{-1}||^{-1}\langle f,f\rangle \leq \sum_{j\in J}v_{j}^{2}\langle \Lambda_{j}\theta P_{W_{j}}Cf,\Lambda_{j}\theta P_{W_{j}}C^{'}f\rangle
			\end{equation}
			Using  \eqref{eq21} and \eqref{eq18} we conclude that
			\begin{equation*}
				A||(\theta^{\ast}\theta)^{-1}||^{-1}\langle f,f\rangle \leq \sum_{j\in J}v_{j}^{2}\langle \Lambda_{j}\theta P_{W_{j}}Cf,\Lambda_{j}\theta P_{W_{j}}C^{'}f\rangle \leq B||\theta ||^{2}\langle f,f \rangle, \qquad \forall f\in H.
			\end{equation*}
			Therefore 	$\{W_{j},\Lambda_{j}\theta,v_{j}\}_{j\in J}$ is a $(C,C^{'})-$controlled $g-$fusion frame for $H$.
		\end{proof}
		\begin{theorem}
			Let $\{W_{j},\Lambda_{j},v_{j}\}_{j\in J}$ be a $(C,C^{'})-$controlled $g-$fusion frame for $H$. with bounds $A$ and $B$. Let $\theta\in End_{\mathcal{A}}^{\ast}(L,H)$ be injective and has a closed range. Suppose that $\theta$ commute with $\Lambda_{j}P_{W_{j}}C$ and $\Lambda_{j}P_{W_{j}}C^{'}$ for all $j\in J$. Then  $\{W_{j},\theta \Lambda_{j},v_{j}\}_{j\in J}$ be a $(C,C^{'})-$controlled $g-$fusion frame for $H$. 
		\end{theorem}
		\begin{proof}
			Let $\{W_{j},\Lambda_{j},v_{j}\}_{j\in J}$ be a $(C,C^{'})-$controlled $g-$fusion frame for $H$ with bounds $A$ and $B$, then 
			\begin{equation*}
				A\langle f,f \rangle \leq \sum_{j\in J}v_{j}^{2}\langle \Lambda_{j}P_{W_{j}}Cf,\Lambda_{j}P_{W_{j}}C^{'}f\rangle\leq B\langle f,f\rangle,\qquad \forall f\in H.
			\end{equation*}
			We have for each $f\in H$
			\begin{align}\label{eq19}
				\sum_{j\in J}v_{j}^{2}\langle \theta \Lambda_{j}P_{W_{j}}Cf,\theta\Lambda_{j}P_{W_{j}}C^{'}f\rangle &\leq ||\theta ||^{2} \sum_{j\in J}v_{j}^{2}\langle \Lambda_{j}P_{W_{j}}Cf,\Lambda_{j}P_{W_{j}}C^{'}f\rangle\notag\\&\leq B||\theta ||^{2}\langle f,f\rangle
			\end{align}
			On the other hand, 
			\begin{equation*}
				A\langle \theta f,\theta f\rangle \leq  \sum_{j\in J}v_{j}^{2}\langle \theta \Lambda_{j}P_{W_{j}}Cf,\theta \Lambda_{j}P_{W_{j}}C^{'}f\rangle= \sum_{j\in J}v_{j}^{2}\langle \Lambda_{j}P_{W_{j}}C\theta f,\Lambda_{j}P_{W_{j}}C^{'}\theta f\rangle
			\end{equation*}
			By lemma \ref{3}, we have
			\begin{equation}\label{eq20}
				A||(\theta^{\ast}\theta)^{-1}||^{-1}\langle f,f \rangle \leq  \sum_{j\in J}v_{j}^{2}\langle \theta\Lambda_{j}P_{W_{j}}Cf,\theta\Lambda_{j}P_{W_{j}}C^{'}f\rangle
			\end{equation}
			Using \eqref{eq19} and \eqref{eq20} , we conclude that 
			\begin{equation*}
				A||(\theta^{\ast}\theta)^{-1}||^{-1}\langle f,f \rangle \leq  \sum_{j\in J}v_{j}^{2}\langle \theta\Lambda_{j}P_{W_{j}}Cf,\theta\Lambda_{j}P_{W_{j}}C^{'}f\rangle\leq B||\theta ||^{2}\langle f,f \rangle,\qquad \forall f\in H. 
			\end{equation*}
			Hence, $\{W_{j},\theta \Lambda_{j},v_{j}\}_{j\in J}$ is a $(C,C^{'})-$controlled $g-$fusion frame for $H$. 
		\end{proof}	
		Under wich conditions a  $(C,C^{'})-$controlled $g-$fusion frame for $H$ with $H$ a $C^{\ast}-$module over a unital $C^{\ast}-$algebras $\mathcal{A}$ is also a  $(C,C^{'})-$controlled $g-$fusion frame for $H$ with $H$ a $C^{\ast}-$module over a unital $C^{\ast}-$algebras $\mathcal{B}$. the following theorem answer this questions. We teak in next theorem $H_{j}\subset H$, $\forall j\in J$.
		\begin{theorem}
			Let $(H,\mathcal{A},\langle .,.\rangle_{\mathcal{A}})$ and $(H,\mathcal{B},\langle .,.\rangle_{\mathcal{B}})$ be two Hilbert $C^{\ast}-$modules and let $\phi :\mathcal{A}\rightarrow \mathcal{B}$ be a $\ast -$homomorphisme and $\theta$ be a map on $H$ such that $\langle \theta f,\theta g\rangle_{\mathcal{B}}=\phi(\langle f,g\rangle_{\mathcal{A}})$ for all $f$, $g\in H$. Suppose that $\Lambda_{CC^{'}}=\{W_{j},\Lambda_{j},v_{j}\}_{j\in J}$ is a $(C,C^{'})-$controlled $g-$fusion frame for $(H,\mathcal{A},\langle .,.\rangle_{\mathcal{A}})$ with frame operator $S_{\mathcal{A}}$ and lower and upper bounds $A$ and $B$ respectively. If $\theta$ is surjective such that $\theta\Lambda_{j}P_{W_{j}}=\Lambda_{j}P_{W_{j}}\theta$ for each $j\in J$ and $\theta C=C\theta$ and $\theta C^{'}=C^{'}\theta$.\\ Then $\{W_{j},\Lambda_{j},\phi(v_{j})\}_{j\in J}$ is a $(C,C^{'})-$controlled $g-$fusion frame for $(H,\mathcal{B},\langle .,.\rangle_{\mathcal{B}})$ with frame operator $S_{\mathcal{B}}$ and lower and upper bounds $A$ and $B$ respectively and $\langle S_{\mathcal{B}}\theta f, \theta g\rangle_{\mathcal{B}}=\phi(\langle S_{\mathcal{A}}f,g\rangle_{\mathcal{A}})$.
		\end{theorem}
		\begin{proof}
			Since $\theta$ is surjective, then for every $g\in H$ there exists $f\in H$ such that $\theta f=g$. Using the definition of $(C,C^{'})-$controlled $g-$fusion frame for $(H,\mathcal{A},\langle .,.\rangle_{\mathcal{A}})$ we have
			\begin{equation*}
				A\langle f,f\rangle_{\mathcal{A}} \leq \sum_{j\in J}v_{j}^{2}\langle \Lambda_{j}P_{W_{j}}Cf,\Lambda_{j}P_{W_{j}}C^{'}f\rangle_{\mathcal{A}}\leq B\langle f,f\rangle_{\mathcal{A}}
			\end{equation*}
			By lemma \ref{ALDE} we have 
			\begin{equation*}
				\phi\big(A\langle f,f\rangle_{\mathcal{A}}\big)\leq \phi\bigg(\sum_{j\in J}v_{j}^{2}\langle \Lambda_{j}P_{W_{j}}Cf,\Lambda_{j}P_{W_{j}}C^{'}f\rangle_{\mathcal{A}}\bigg)\leq\phi\big(B\langle f,f\rangle_{\mathcal{A}}\big)
			\end{equation*}
			Frome the definition of $\ast-$homomorphisme we have
			\begin{equation*}
				A\phi\big(\langle f,f\rangle_{\mathcal{A}}\big)\leq\sum_{j\in J}\phi(v_{j}^{2})\phi\bigg(\langle \Lambda_{j}P_{W_{j}}Cf,\Lambda_{j}P_{W_{j}}C^{'}f\rangle_{\mathcal{A}}\bigg)\leq B\phi\big(\langle f,f\rangle_{\mathcal{A}}\big)
			\end{equation*}
			Using the relation betwen $\theta$ and $\phi$ we get 
			\begin{equation*}
				A\langle \theta f,\theta f\rangle_{\mathcal{B}}\leq\sum_{j\in J}\phi(v_{j})^{2}\langle \theta\Lambda_{j}P_{W_{j}}Cf,\theta\Lambda_{j}P_{W_{j}}C^{'}f\rangle_{\mathcal{B}}\leq B\langle \theta f,\theta f\rangle_{\mathcal{B}}
			\end{equation*}
			Since $\theta\Lambda_{j}P_{W_{j}}=\Lambda_{j}P_{W_{j}}\theta$ for each $j\in J$ and $\theta C=C\theta$ and $\theta C^{'}=C^{'}\theta$  we have
			\begin{equation*}
				A\langle\theta f ,\theta f\rangle_{\mathcal{B}}\leq\sum_{j\in J}\phi(v_{j})^{2}\langle \Lambda_{j}P_{W_{j}}C\theta f,\Lambda_{j}P_{W_{j}}C^{'}\theta f\rangle_{\mathcal{B}}\leq B\langle \theta f,\theta f\rangle_{\mathcal{B}}
			\end{equation*}
			Therefore,
			\begin{equation*}
				A\langle g,g\rangle_{\mathcal{B}}\leq\sum_{j\in J}\phi(v_{j})^{2}\langle \Lambda_{j}P_{W_{j}}Cg,\Lambda_{j}P_{W_{j}}C^{'}g\rangle_{\mathcal{B}}\leq B\langle g,g\rangle_{\mathcal{B}},\qquad \forall g\in H.
			\end{equation*}
			This implies that $\{W_{j},\Lambda_{j},\phi(v_{j})\}_{j\in J}$ is a $(C,C^{'})-$controlled $g-$fusion frame for $(H,\mathcal{B},\langle .,.\rangle_{\mathcal{B}})$ with bounds $A$ and $B$.
			Moreover we have
			\begin{align*}
				\phi\big(\langle S_{\mathcal{A}}f,g\rangle_{\mathcal{A}}\big)&=\phi\bigg(\langle\sum_{j\in J}v_{j}^{2}C^{'}P_{W_{j}}\Lambda_{j}^{\ast}\Lambda_{j}P_{W_{j}}Cf,g\rangle_{\mathcal{A}}\bigg)\\&=\sum_{j\in J}\phi(v_{j})^{2}\phi\bigg(\langle\Lambda_{j}P_{W_{j}}Cf,\Lambda_{j}P_{W_{j}}C^{'}g\rangle_{\mathcal{A}}\bigg)\\&=\sum_{j\in J}\phi(v_{j})^{2}\langle\theta\Lambda_{j}P_{W_{j}}Cf,\theta\Lambda_{j}P_{W_{j}}C^{'}g\rangle_{\mathcal{B}}\\&=\langle\sum_{j\in J}\phi(v_{j})^{2}C^{'}P_{W_{j}}\Lambda_{j}^{\ast}\Lambda_{j}P_{W_{j}}C\theta f,\theta g\rangle_{\mathcal{B}}\\&=\langle S_{\mathcal{B}}\theta f,\theta g\rangle_{\mathcal{B}}.
			\end{align*}
		\end{proof}
		\section{$(C,C^{'})-$controlled $K-g-$fusion frames in Hilbert $C^{\ast}-$modules}
		Firstly we give the definition of $K-g-$fusion frame in Hilbert $C^{\ast}-$modules.
			\begin{definition}\cite{rfn1}
			Let $\mathcal{A}$ be a unital $C^{\ast}-$algebra and $\mathcal{H}$ be a countably generated Hilbert $\mathcal{A}-$module. let $\left( v_{j}\right) _{j\in\mathrm{J}}$ be a family of weights in $\mathcal{A}$,i.e.,each $ v_{j}$ is a positive invertible element frome the center of $\mathcal{A}$, let $\left( W_{j}\right) _{j\in\mathrm{J}}$ be a collection of orthogonally complemented closed submodules of $\mathcal{H}$. Let $\left( \mathcal{K}_{j}\right) _{j\in\mathrm{J}}$ a sequence of closed submodules of $\mathcal{K}$ and $\large{\Lambda }_{j}\in End_{\mathcal{A}}^{\ast}(\mathcal{H},\mathcal{K}_{j}) $ for each $j\in\mathrm{J}$ and $K\in End^{\ast}_{\mathcal{A}}(\mathcal{H})$. We say $\Lambda=(W_{j},\Lambda_{j},v_{j})_{j\in\mathrm{J}}$ is $K-g-$fusion frame for $\mathcal{H}$
			with respect to $(\mathcal{K}_{j})_{j\in\mathrm{J}}$ if there exist real constants $ 0 < A \leq B<\infty$ such that  
			\begin{equation}\label{eq7}
				A\langle K^{\ast}f,K^{\ast}f\rangle \leq \sum_{j\in \mathrm{J}}v_{j}^{2} \langle \Lambda_{j} P_{W_{j}}f ,  \Lambda_{j} P_{W_{j}}f \rangle \leq B\langle f, f\rangle, \qquad \forall f\in \mathcal{H}.
			\end{equation} 
			The constants $A$ and $B$ are called a lower and upper bounds of $K-g-$fusion frame, respectively. If the left-hand inequality of \eqref{eq7} is an equality, we say that $\Lambda$ is a tight $K-g-$fusion frame. If $K=I_{\mathcal{H}}$ then $\Lambda$ is a $g-$fusion frame and if $K=I_{\mathcal{H}}$ and $\Lambda_{j}=P_{W_{j}}$ for any $ j\in \mathrm{J}$, then $\Lambda$ is a fusion frame for $\mathcal{H}$
		\end{definition}
		\begin{definition}
			Let $C$, $C^{'}\in GL^{+}(H)$ and $K\in End_{\mathcal{A}}^{\ast}(H)$. $\{W_{j}\}_{j\in J}$ be a sequence of closed submodules orthogonally complemented of $H$, $\{v_{j}\}_{j\in J}$ be a family of weights in $\mathcal{A}$, i.e., each $v_{j}$ is a positive invertible element frome the center of $\mathcal{A}$ and $\Lambda_{j}\in End_{A}^{\ast}(H,H_{j})$ for each $j\in J$. We say $\Lambda_{CC^{'}}=\{W_{j},\Lambda_{j},v_{j}\}_{j\in J}$ is a $(C,C^{'})-$controlled $K-g-$fusion frame for $H$ if there exists $0<A\leq B<\infty$ such that 
			\begin{equation}\label{eq10}
				A\langle K^{\ast}f,K^{\ast}f \rangle \leq \sum_{j\in J}v_{j}^{2}\langle \Lambda_{j}P_{W_{j}}Cf,\Lambda_{j}P_{W_{j}}C^{'}f\rangle\leq B\langle f,f\rangle,\qquad \forall f\in H.
			\end{equation}
			The constants $A$ and $B$ are called a lower and upper bounds of $(C,C^{'})-$controlled $K-g-$fusion frame, respectively. If the left-hand inequality of \eqref{eq10}  is an equality, we say that $\Lambda_{CC^{'}}$ is a tight $(C,C^{'})-$controlled $K-g-$fusion frame for $H$.
		\end{definition}
		\begin{remark}
			If $\Lambda_{CC^{'}}$ is a $(C,C^{'})-$controlled $K-g-$fusion frame for $H$ with bounds $A$ and $B$ we have 
			\begin{equation}\label{eq11}
				AKK^{\ast}\leq S_{(C,C^{'})} \leq BI_{H.}
			\end{equation}
		\end{remark}
		From  equality \eqref{eq12} and inequality \eqref{eq11} we have  
		\begin{proposition}\label{1}
			Let $K\in End_{\mathcal{A}}^{\ast}(H)$, and $\Lambda_{CC^{'}}$ be a $(C,C^{'})-$controlled $g-$fusion bessel sequence for $H$. Then $\Lambda_{CC^{'}}$ is a $(C,C^{'})-$controlled $K-g-$fusion frame for $H$ if and only if there exist a constant $A>0$ such that $AKK^{\ast}\leq S_{(C,C^{'})}$ where $S_{(C,C^{'})}$ is the frame operator for $\Lambda_{CC^{'}}$. 
		\end{proposition}
		\begin{theorem}
			Let $\Lambda_{CC^{'}}=\{W_{j},\Lambda_{j},v_{j}\}_{j\in J}$ and $\Gamma_{CC^{'}}=\{V_{j},\Gamma_{j},u_{j}\}_{j\in J}$ be two  $(C,C^{'})-$controlled $g-$fusion bessel sequences for $H$ with bounds $B_{1}$ and $B_{2}$, respectively. Suppose that $T_{\Lambda_{CC^{'}}}$ and $T_{\Gamma_{CC^{'}}}$ be their synthesis operators such that $T_{\Gamma_{CC^{'}}}T_{\Lambda_{CC^{'}}}^{\ast}=K^{\ast}$ for some $K\in End_{\mathcal{A}}^{\ast}(H)$. Then, both $\Lambda_{CC^{'}}$ and $\Gamma_{CC^{'}}$ are $(C,C^{'})-$controlled $K$ and $K^{\ast}-g-$fusion frames for $H$, respectively. 
		\end{theorem}
		\begin{proof}
			For each $f\in H$, we have
			\begin{align*}
				\langle K^{\ast}f,K^{\ast}f\rangle=\langle T_{\Gamma_{CC^{'}}}T_{\Lambda_{CC^{'}}}^{\ast}f,T_{\Gamma_{CC^{'}}}T_{\Lambda_{CC^{'}}}^{\ast}f\rangle &\leq ||T_{\Gamma_{CC^{'}}}||^{2}\langle T_{\Lambda_{CC^{'}}}^{\ast},T_{\Lambda_{CC^{'}}}^{\ast}f\rangle\\&\leq B_{2}\sum_{j\in J}v_{j}^{2}\langle \Lambda_{j}P_{W_{j}}Cf,\Lambda_{j}P_{W_{j}}C^{'}f\rangle,
			\end{align*}
			Hence
			\begin{equation*}
				B_{2}^{-1}	\langle K^{\ast}f,K^{\ast}f\rangle\leq\sum_{j\in J}v_{j}^{2}\langle \Lambda_{j}P_{W_{j}}Cf,\Lambda_{j}P_{W_{j}}C^{'}f\rangle.
			\end{equation*}
			This means that $\Lambda_{CC^{'}}$ is a  $(C,C^{'})-$controlled $K-g-$fusion frame for $H$. Similarly, $\Gamma_{CC^{'}}$ is a $(C,C^{'})-$controlled $K^{\ast}-g-$fusion frame for $H$ with the lower bound $B_{1}^{-1}$.
		\end{proof}
		\begin{theorem} Let $U\in End_{\mathcal{A}}^{\ast}(H)$ be an invertible operator on $H$ and $\Lambda_{CC^{'}}=\{W_{j},\Lambda_{j},v_{j}\}_{j\in J}$ be a $(C,C^{'})-$controlled $K-g-$fusion frame for $H$ for some $K\in End_{\mathcal{A}}^{\ast}(H)$. Suppose that $U^{\ast}UW_{j}\subset W_{j}$, $\forall j\in J$ and $C$, $C^{'}$ commute with $U$. Then $\Gamma_{CC^{'}}=\{UW_{j}, \Lambda_{j}P_{W_{j}}U^{\ast},v_{j}\}_{j\in J}$ is a  $(C,C^{'})-$controlled $UKU^{\ast}-g-$fusion frame for $H$.
		\end{theorem}
		\begin{proof} 
			Since $\Lambda_{CC^{'}}$ is a  $(C,C^{'})-$controlled $K-g-$fusion frame for $H$, $\exists$ $A,B > 0$	such that  
			\begin{equation*}
				A\langle K^{\ast}f,K^{\ast}f \rangle \leq \sum_{j\in J}v_{j}^{2}\langle \Lambda_{j}P_{W_{j}}Cf,\Lambda_{j}P_{W_{j}}C^{'}f \rangle \leq B\langle f, f \rangle, \qquad \forall f \in H. 
			\end{equation*} 
			Also, $U$ is an invertible linear operator on $H$, so for any $j \in J$, $UW_{j}$ is closed in $H$. Now, for each $f\in H$, using lemma \ref{9}, we obtain
			\begin{align*}
				\sum_{j\in J}v_{j}^{2}\langle \Lambda_{j}P_{W_{j}}U^{\ast}P_{UW_{j}}Cf,\Lambda_{j}P_{W_{j}}U^{\ast}P_{UW_{j}}C^{'}f \rangle &= \sum_{j\in J}v_{j}^{2} \langle \Lambda_{j}P_{W_{j}}U^{\ast}Cf, \Lambda_{j}P_{W_{j}}U^{\ast}C^{'}f \rangle\\&=\sum_{j\in J}v_{j}^{2} \langle \Lambda_{j}P_{W_{j}}CU^{\ast}f, \Lambda_{j}P_{W_{j}}C^{'}U^{\ast}f \rangle 
				\\&\leq B \langle U^{\ast}f, U^{\ast}f \rangle \\&\leq B||U||^{2}\langle f, f \rangle.  
			\end{align*}
			On the other hand, for each $f\in H$ 
			\begin{align*}
				A\langle (UKU^{\ast})^{\ast}f, (UKU^{\ast})^{\ast}f \rangle &=A \langle UK^{\ast}U^{\ast}f, UK^{\ast}U^{\ast}f \rangle \\&\leq A||U||^{2} \langle K^{\ast}U^{\ast}f, K^{\ast}U^{\ast}f \rangle
				\\& \leq ||U||^{2} \sum_{j\in J}v_{j}^{2}\langle \Lambda_{j}P_{W_{j}}C(U^{\ast}f),\Lambda_{j}P_{W_{j}}C^{'}(U^{\ast}f) \rangle\\&=||U||^{2} \sum_{j\in J}v_{j}^{2}\langle \Lambda_{j}P_{W_{j}}U^{\ast}Cf,\Lambda_{j}P_{W_{j}}U^{\ast}C^{'}f \rangle \\&\leq ||U||^{2} \sum_{j\in J}v_{j}^{2}\langle \Lambda_{j}P_{W_{j}}U^{\ast}P_{UW_{j}}Cf,\Lambda_{j}P_{W_{j}}U^{\ast}P_{UW_{j}}C^{'}f \rangle,
			\end{align*}
			Then 
			\begin{equation*}
				\frac{A}{||U||^{2}} \langle (UKU^{\ast})^{\ast}f, (UKU^{\ast})^{\ast}f \rangle \leq \sum_{j\in J}v_{j}^{2}\langle \Lambda_{j}P_{W_{j}}U^{\ast}P_{UW_{j}}Cf,\Lambda_{j}P_{W_{j}}U^{\ast}P_{UW_{j}}C^{'}f\rangle
			\end{equation*}
			Therefore, $\Gamma_{CC^{'}}$ is a $(C,C^{'})-$controlled $UKU^{\ast}-g-$fusion frame for $H$.
		\end{proof}
		\begin{theorem}
			Let  $U\in End_{\mathcal{A}}^{\ast}(H)$ be an invertible operator on $H$ and  $\Gamma_{CC^{'}}=\{UW_{j}, \Lambda_{j}P_{W_{j}}U^{\ast},v_{j}\}_{j\in J}$ be a $(C,C^{'})-$controlled $K-g-$fusion frame for $H$ for some  $K\in End_{\mathcal{A}}^{\ast}(H)$. Suppose that $U^{\ast}UW_{j}\subset W_{j}$, $\forall j\in J$ and $C$, $C^{'}$ commute with $U$. Then  $\Lambda_{CC^{'}}=\{W_{j},\Lambda_{j},v_{j}\}_{j\in J}$ is a $(C,C^{'})-$controlled $U^{-1}KU-g-$fusion frame for $H$. 
		\end{theorem}
		\begin{proof}
			Since $\Gamma_{CC^{'}}=\{UW_{j},\Lambda_{j}P_{W_{j}}U^{\ast},v_{j}\}_{j\in J}$ is a  $(C,C^{'})-$controlled $K-g-$fusion frame for $H$, $\exists$ $A$, $B>0$ such that
			\begin{equation*}
				A\langle K^{\ast}f,K^{\ast}f\rangle \leq \sum_{j\in J}v_{j}^{2}\langle \Lambda_{j}P_{W_{j}}U^{\ast}P_{UW_{j}}Cf,\Lambda_{j}P_{W_{j}}U^{\ast}P_{UW_{j}}C^{'}f\rangle \leq B\langle f,f \rangle.\qquad\forall f\in H.
			\end{equation*}
			Let $f \in H$, we have 
			\begin{align*}
				A\langle (U^{-1}KU)^{\ast}f,(U^{-1}KU)^{\ast}f \rangle &=A\langle U^{\ast}K^{\ast}(U^{-1})^{\ast}f,U^{\ast}K^{\ast}(U^{-1})^{\ast}f\rangle  
				\\&\leq A||U||^{2}\langle K^{\ast}(U^{-1})^{\ast}f,K^{\ast}(U^{-1})^{\ast}f\rangle \\&\leq ||U||^{2}  \sum_{j\in J}v_{j}^{2}\langle \Lambda_{j}P_{W_{j}}U^{\ast}P_{UW_{j}}C(U^{-1})^{\ast}f,\Lambda_{j}P_{W_{j}}U^{\ast}P_{UW_{j}}C^{'}(U^{-1})^{\ast}f \rangle 
				\\&\leq ||U||^{2}  \sum_{j\in J}v_{j}^{2}\langle \Lambda_{j}P_{W_{j}}U^{\ast}C(U^{-1})^{\ast}f,\Lambda_{j}P_{W_{j}}U^{\ast}C^{'}(U^{-1})^{\ast}f \rangle \\&= ||U||^{2}  \sum_{j\in J}v_{j}^{2}\langle \Lambda_{j}P_{W_{j}}U^{\ast}(U^{-1})^{\ast}Cf,\Lambda_{j}P_{W_{j}}U^{\ast}(U^{-1})^{\ast}C^{'}f \rangle \\&=  ||U||^{2}  \sum_{j\in J}v_{j}^{2}\langle \Lambda_{j}P_{W_{j}}Cf,\Lambda_{j}P_{W_{j}}C^{'}f \rangle.  
			\end{align*}
			Then, for each $f\in H$, we have  
			\begin{equation*}
				\frac{A}{||U||^{2}}\langle (U^{-1}KU)^{\ast}f,(U^{-1}KU)^{\ast}f \rangle \leq \sum_{j\in \mathrm{J}}v_{j}^{2}\langle \Lambda_{j}P_{W_{j}}Cf,\Lambda_{j}P_{W_{j}}C^{'}f \rangle.  
			\end{equation*}
			Also, for each $f\in H$, we have 
			\begin{align*}
				\sum_{j\in J}v_{j}^{2}\langle \Lambda_{j}P_{W_{j}}Cf,\Lambda_{j}P_{W_{j}}C^{'}f \rangle &= \sum_{j\in J}v_{j}^{2}\langle \Lambda_{j}P_{W_{j}}CU^{\ast}(U^{-1})^{\ast}f,\Lambda_{j}P_{W_{j}}C^{'}U^{\ast}(U^{-1})^{\ast}f \rangle  
				\\&=\sum_{j\in J}v_{j}^{2}\langle \Lambda_{j}P_{W_{j}}U^{\ast}C(U^{-1})^{\ast}f,\Lambda_{j}P_{W_{j}}U^{\ast}C^{'}(U^{-1})^{\ast}f \rangle \\&=\sum_{j\in J}v_{j}^{2}\langle \Lambda_{j}P_{W_{j}}U^{\ast}P_{UW_{j}}C(U^{-1})^{\ast}f,\Lambda_{j}P_{W_{j}}U^{\ast}P_{UW_{j}}C^{'}(U^{-1})^{\ast}f \rangle \\&\leq B\langle (U^{-1})^{\ast}f,(U^{-1})^{\ast}f \rangle \\&\leq B ||U^{-1}||^{2}\langle f,f \rangle.  
			\end{align*}
			Thus, $\Lambda_{CC^{'}}$ is a $(C,C^{'})-$controlled $U^{-1}KU-g-$fusion frame for $H$.
		\end{proof}
		\begin{theorem}\label{10}
			Let  $K\in End_{\mathcal{A}}^{\ast}(H)$ be an invertible operator on $H$ and $\Lambda_{CC^{'}}=\{W_{j},\Lambda_{j},v_{j}\}_{j\in J}$ be a  $(C,C^{'})-$controlled $g-$fusion frame for $H$ with frame bounds $A$, $B$ and $S_{(C,C^{'})}$ be the associated $(C,C^{'})-$controlled $g-$fusion frame operator. Suppose that for all $j\in J$, $T^{\ast}TW_{j}\subset W_{j}$, where $T=KS_{(C,C^{'})}^{-1}$ and $C$, $C^{'}$ commute with $T$. Then $\{KS_{(C,C^{'})}^{-1}W_{j},\Lambda_{j}P_{W_{j}}S_{(C,C^{'})}^{-1}K^{\ast},v_{j}\}_{j\in J}$ is a $(C,C^{'})-$controlled $K-g-$fusion frame for $H$ with the corresponding  $(C,C^{'})-$controlled $g-$fusion frame operator $KS_{(C,C^{'})}^{-1}K^{\ast}$.
		\end{theorem}
		\begin{proof}
			We now $T=KS_{(C,C^{'})}^{-1}$ is invertible on $H$ and $T^{\ast}=(KS_{(C,C^{'})}^{-1})^{\ast}=S_{(C,C^{'})}^{-1}K^{\ast}$. For each $f \in H$, we have 
			\begin{align*}
				\langle K^{\ast}f,K^{\ast}f \rangle &= \langle S_{(C,C^{'})}S_{(C,C^{'})}^{-1}K^{\ast}f,S_{(C,C^{'})}S_{(C,C^{'})}^{-1}K^{\ast}f \rangle \\&\leq ||S_{(C,C^{'})}||^{2} \langle S_{(C,C^{'})}^{-1}K^{\ast}f,S_{(C,C^{'})}^{-1}K^{\ast}f\rangle 
				\\&\leq B^{2}\langle S_{(C,C^{'})}^{-1}K^{\ast}f,S_{(C,C^{'})}^{-1}K^{\ast}f\rangle. 
			\end{align*}
			Now for each $f \in H$, we get
			\begin{align*}
				\sum_{j\in J}v_{j}^{2}\langle \Lambda_{j}P_{W_{j}}T^{\ast}P_{TW_{j}}C(f),\Lambda_{j}P_{W_{j}}T^{\ast}P_{TW_{j}}C^{'}(f) \rangle &=\sum_{j\in J}v_{j}^{2}\langle \Lambda_{j}P_{W_{j}}T^{\ast}C(f),\Lambda_{j}P_{W_{j}}T^{\ast}C^{'}(f) \rangle 
				\\&=\sum_{j\in J}v_{j}^{2}\langle \Lambda_{j}P_{W_{j}}CT^{\ast}(f),\Lambda_{j}P_{W_{j}}C^{'}T^{\ast}(f) \rangle\\&\leq B\langle T^{\ast}f,T^{\ast}f\rangle \\&\leq B ||T||^{2}\langle f,f\rangle \\&\leq B ||S_{(C,C^{'})}^{-1}||^{2}||K||^{2}\langle f,f\rangle \\&\leq \frac{B}{A^{2}}||K||^{2}\langle f,f\rangle. 
			\end{align*}
			On the other hand, for each $f\in H$, we have
			\begin{align*}
				\sum_{j\in J}v_{j}^{2}\langle \Lambda_{j}P_{W_{j}}T^{\ast}P_{TW_{j}}C(f),\Lambda_{j}P_{W_{j}}T^{\ast}P_{TW_{j}}C^{'}(f) \rangle   &=\sum_{j\in J}v_{j}^{2}\langle \Lambda_{j}P_{W_{j}}T^{\ast}C(f),\Lambda_{j}P_{W_{j}}T^{\ast}C^{'}(f) \rangle 
				\\&=\sum_{j\in J}v_{j}^{2}\langle \Lambda_{j}P_{W_{j}}CT^{\ast}(f),\Lambda_{j}P_{W_{j}}C^{'}T^{\ast}(f) \rangle\\&\geq A \langle T^{\ast}f, T^{\ast}f \rangle \\&= A \langle S_{(C,C^{'})}^{-1}K^{\ast}f,S_{(C,C^{'})}^{-1}K^{\ast}f \rangle \\&\geq \frac{A}{B^{2}}\langle K^{\ast}f,K^{\ast}f \rangle. 
			\end{align*}
			Thus $\{KS_{(C,C^{'})}^{-1}W_{j},\Lambda_{j}P_{W_{j}}S_{(C,C^{'})}^{-1}K^{\ast},v_{j}\}_{j\in J}$ is a $(C,C^{'})-$controlled $K-g-$fusion frame for $H$.\\
			For each $f\in H$, we have 
			\begin{align*}
				\sum_{j\in J}v_{j}^{2}C^{'}P_{TW_{j}}(\Lambda_{j}P_{W_{j}}T^{\ast})^{\ast}(\Lambda_{j}P_{W_{j}}T^{\ast})P_{TW_{j}}Cf &= \sum_{j\in J}v_{j}^{2}C^{'}P_{TW_{j}}TP_{W_{j}}\Lambda_{j}^{\ast}(\Lambda_{j}P_{W_{j}}T^{\ast})P_{TW_{j}}Cf
				\\&=\sum_{j\in \mathrm{J}}v_{j}^{2}C^{'}(P_{W_{j}}T^{\ast}P_{TW_{j}})^{\ast}\Lambda_{j}^{\ast}\Lambda_{j}(P_{W_{j}}T^{\ast}P_{TW{j}})Cf \\&= \sum_{j\in J}v_{j}^{2}C^{'}TP_{W_{j}}\Lambda_{j}^{\ast}\Lambda_{j}P_{W_{j}}T^{\ast}Cf \\&=\sum_{j\in J}v_{j}^{2}TC^{'}P_{W_{j}}\Lambda_{j}^{\ast}\Lambda_{j}P_{W_{j}}CT^{\ast}f\\&= T(\sum_{j\in J}v_{j}^{2}C^{'}P_{W_{j}}\Lambda_{j}^{\ast}\Lambda_{j}P_{W_{j}}CT^{\ast}f) 
				\\&= TS_{(C,C^{'})}T^{\ast}(f) = KS_{(C,C^{'})}^{-1}K^{\ast}(f).
			\end{align*}
			This implies that $KS_{(C,C^{'})}^{-1}K^{\ast}$ is the associated $(C,C^{'})-$controlled $g-$fusion frame operator.
		\end{proof}
		
		The next theorem we give an equivqlent definition of $(C,C^{'})-$controlled $K-g-$fusion frame.
		\begin{theorem}\label{8}
			Let $K\in End_{\mathcal{A}}^{\ast}(H)$. Then $\Lambda_{CC^{'}}$ is a $(C,C^{'})-$controlled $K-g-$fusion frame for $H$ if and only if there exist constants $A$, $B>0$ such that
			\begin{equation}\label{eq14}
				A||K^{\ast}f||^{2}\leq ||\sum_{j\in J}v_{j}^{2}\langle \Lambda_{j}P_{W_{j}}Cf,\Lambda_{j}P_{W_{j}}C^{'}f\rangle||\leq B||f||^{2},\qquad\forall f\in H.
			\end{equation} 
		\end{theorem}
		\begin{proof}
			Evidently, every $(C,C^{'})-$controlled $K-g-$fusion frame for $H$ satisfies \eqref{eq14}.\\
			For the converse, we suppose that \eqref{eq14} holds. 
			For any $\{f_{j}\}_{j\in J}\in l^{2}(\{H_{j}\}_{j\in J})$,
			\begin{align*}
				|| \sum_{j\in J}v_{j}(CC^{'})^{\frac{1}{2}}P_{W_{j}}\Lambda_{j}^{\ast}f_{j}||&=\sup_{||g||=1}||\langle \sum_{j\in J}v_{j}(CC^{'})^{\frac{1}{2}}P_{W_{j}}\Lambda_{j}^{\ast}f_{j},g\rangle||\\&=\sup_{||g||=1}||\sum_{j\in J}\langle v_{j}(CC^{'})^{\frac{1}{2}}P_{W_{j}}\Lambda_{j}^{\ast}f_{j},g\rangle||\\&=\sup_{||g||=1}||\sum_{j\in J}\langle f_{j}, v_{j}\Lambda_{j}P_{W_{j}}(CC^{'})^{\frac{1}{2}}g\rangle||\\&\leq \sup_{||g||=1}||\sum_{j\in J}\langle f_{j}, f_{j}\rangle||^{\frac{1}{2}}||\sum_{j\in J}v_{j}^{2}\langle\Lambda_{j}P_{W_{j}}(CC^{'})^{\frac{1}{2}}g, \Lambda_{j}P_{W_{j}}(CC^{'})^{\frac{1}{2}}g\rangle||^{\frac{1}{2}}\\&= \sup_{||g||=1}||\sum_{j\in J}\langle f_{j}, f_{j}\rangle||^{\frac{1}{2}}||\sum_{j\in J}v_{j}^{2}\langle\Lambda_{j}P_{W_{j}}Cg, \Lambda_{j}P_{W_{j}}C^{'}g\rangle||^{\frac{1}{2}}\\&\leq \sup_{||g||=1}||\sum_{j\in J}\langle f_{j}, f_{j}\rangle||^{\frac{1}{2}}\sqrt{B}||g||=\sqrt{B}||\{f_{j}\}_{j\in J}||.
			\end{align*}
			Thus the series $ \sum_{j\in J}v_{j}(CC^{'})^{\frac{1}{2}}P_{W_{j}}\Lambda_{j}^{\ast}f_{j}$ converges in $H$ unconditionally. Since 
			\begin{equation*}
				\langle Tf, \{f_{j}\}_{j\in J}\rangle=\sum_{j\in J}\langle v_{j}\Lambda_{j}P_{W_{j}}(CC^{'})^{\frac{1}{2}}f,f_{j}\rangle=\langle f,\sum_{j\in J}v_{j}(CC^{'})^{\frac{1}{2}}P_{W_{j}}\Lambda_{j}^{\ast}f_{j}\rangle.
			\end{equation*}
			$T$ is adjointable. Now for each $f\in H$ we have
			\begin{equation*}
				\langle Tf,Tf\rangle=\sum_{j\in J}v_{j}^{2}\langle \Lambda_{j}P_{W_{j}}Cf,\Lambda_{j}P_{W_{j}}C^{'}f\rangle\leq ||T||^{2}\langle f,f\rangle.
			\end{equation*}
			On the other hand the left-hand inequality of \eqref{eq14} gives
			\begin{equation*}
				||K^{\ast}f||^{2}\leq \frac{1}{A}||Tf||^{2},\qquad\forall f\in H.
			\end{equation*}
			Then the lemma \ref{5} implies that there exist a constant $\mu >0$ such that 
			\begin{equation*}
				KK^{\ast}\leq \mu T^{\ast}T,
			\end{equation*}
			And hence 
			\begin{equation*}
				\frac{1}{\mu}\langle K^{\ast}f,K^{\ast}f\rangle\leq \langle Tf,Tf\rangle=\sum_{j\in J}v_{j}^{2}\langle \Lambda_{j}P_{W_{j}}Cf,\Lambda_{j}P_{W_{j}}C^{'}f\rangle,\qquad\forall f\in H.
			\end{equation*}
			Consequently, $\Lambda_{CC^{'}}$ is a $(C,C^{'})-$controlled $K-g-$fusion frame for $H$.
		\end{proof}
		\section{perturbation of $(C,C^{'})-$controlled $K-g-$fusion frame in Hilbert $C^{\ast}-$modules}
		\begin{theorem}
			Let $\Lambda_{CC^{'}}=\{W_{j},\Lambda_{j},v_{j}\}_{j\in J}$ be a $(C,C^{'})-$controlled $K-g-$fusion frame for $H$ with frame bounds $A$, $B$ and $\Gamma_{j}\in End_{\mathcal{A}}^{\ast}(H,H_{j})$. Suppose that for each $f\in H$,
				\begin{align*}
					||((v_{j}\Lambda_{j}P_{W_{j}}-u_{j}\Gamma_{j}P_{V_{j}})(CC^{'})^{\frac{1}{2}}f)_{j\in\mathrm{J}}||\leq &\lambda_{1}||(v_{j}\Lambda_{j}P_{W_{j}}(CC^{'})^{\frac{1}{2}}f)_{j\in\mathrm{J}}||+\\&\lambda_{2}||(u_{j}\Gamma_{j}P_{V_{j}}(CC^{'})^{\frac{1}{2}}f)_{j\in\mathrm{J}}||+\epsilon||K^{\ast}f||.
				\end{align*} 
				where $0<\lambda_{1},\lambda_{2}<1$ and $\epsilon >0$ such that $\epsilon < (1-\lambda_{1})\sqrt{A}$.
		
			Then $\{W_{j},\Gamma_{j},u_{j}\}_{j\in J}$ is a $(C,C^{'})-$controlled $K-g-$fusion frame for $H$.
		\end{theorem}
		\begin{proof}
			We have for each $f\in\mathcal{H}$ 
			\begin{align*}
				||\sum_{j\in \mathrm{J}}u_{j}^{2}\langle \Gamma_{j}P_{V_{j}}Cf,\Gamma_{j}P_{V_{j}}C^{'}f\rangle||^{\frac{1}{2}}&=||(u_{j}\Gamma_{j}P_{V_{j}}(CC^{'})^{\frac{1}{2}}f)_{j\in J}||\\&=||(u_{j}\Gamma_{j}P_{V_{j}}(CC^{'})^{\frac{1}{2}}f)_{j\in J}+(v_{j}\Lambda_{j}P_{W_{j}}(CC^{'})^{\frac{1}{2}}f)_{j\in J}\\&\qquad\qquad-(v_{j}\Lambda_{j}P_{W_{j}}(CC^{'})^{\frac{1}{2}}f)_{j\in J}||\\&\leq ||((u_{j}\Gamma_{j}P_{V_{j}}-v_{j}\Lambda_{j}P_{W_{j}})(CC^{'})^{\frac{1}{2}}f)_{j\in J}||\\&\qquad\qquad+||(v_{j}\Lambda_{j}P_{W_{j}}(CC^{'})^{\frac{1}{2}}f)_{j\in J}||\\&\leq (\lambda_{1}+1)||(v_{j}\Lambda_{j}P_{W_{j}}(CC^{'})^{\frac{1}{2}}f)_{j\in J}||\\&\qquad\qquad+\lambda_{2}||(u_{j}\Gamma_{j}P_{V_{j}}(CC^{'})^{\frac{1}{2}}f)_{j\in J}||+\epsilon||K^{\ast}f||.
			\end{align*}
			So 
			\begin{align*}
				(1-\lambda_{2})||(u_{j}\Gamma_{j}P_{V_{j}}(CC^{'})^{\frac{1}{2}}f)_{j\in J}||\leq (\lambda_{1}+1)\sqrt{B}||f||+\epsilon||K^{\ast}f||.
			\end{align*}
			Then 
			\begin{align*}
				||(u_{j}\Gamma_{j}P_{V_{j}}(CC^{'})^{\frac{1}{2}}f)_{j\in J}||&\leq \frac{(\lambda_{1}+1)\sqrt{B}||f||+\epsilon||K^{\ast}f||}{1-\lambda_{2}}\\&\leq (\frac{(\lambda_{1}+1)\sqrt{B}+\epsilon||K||}{1-\lambda_{2}})||f||.
			\end{align*} 
			Hence 
			\begin{equation*}
				||\sum_{j\in \mathrm{J}}u_{j}^{2}\langle \Gamma_{j}P_{V_{j}}Cf,\Gamma_{j}P_{V_{j}}C^{'}f\rangle||\leq  (\frac{(\lambda_{1}+1)\sqrt{B}+\epsilon||K||}{1-\lambda_{2}})^{2}||f||^{2}. 
			\end{equation*}
			On the other hand for each $f\in\mathcal{H}$
			\begin{align*}
				||\sum_{j\in \mathrm{J}}u_{j}^{2}\langle \Gamma_{j}P_{V_{j}}Cf,\Gamma_{j}P_{V_{j}}C^{'}f\rangle||^{\frac{1}{2}}&=||(u_{j}\Gamma_{j}P_{V_{j}}(CC^{'})^{\frac{1}{2}}f)_{j\in J}||\\&=||((u_{j}\Gamma_{j}P_{V_{j}}-v_{j}\Lambda_{j}P_{W_{j}})(CC^{'})^{\frac{1}{2}}f)_{j\in J}\\&\qquad\qquad+(v_{j}\Lambda_{j}P_{W_{j}}(CC^{'})^{\frac{1}{2}}f)_{j\in J}||\\&\geq||(v_{j}\Lambda_{j}P_{W_{j}}(CC^{'})^{\frac{1}{2}}f)_{j\in J}||\\&\qquad\qquad- ||((u_{j}\Gamma_{j}P_{V_{j}}-v_{j}\Lambda_{j}P_{W_{j}})(CC^{'})^{\frac{1}{2}}f)_{j\in J}||\\&\geq (1-\lambda_{1})||(v_{j}\Lambda_{j}P_{W_{j}}(CC^{'})^{\frac{1}{2}}f)_{j\in J}||\\&\qquad\qquad-\lambda_{2}||(u_{j}\Gamma_{j}P_{V_{j}}(CC^{'})^{\frac{1}{2}}f)_{j\in J}||-\epsilon||K^{\ast}f||.
			\end{align*}
			Hence 
			\begin{equation*}
				||\sum_{j\in \mathrm{J}}u_{j}^{2}\langle \Gamma_{j}P_{V_{j}}Cf,\Gamma_{j}P_{V_{j}}C^{'}f\rangle||\geq (\frac{(1-\lambda_{1})\sqrt{A}-\epsilon}{1+\lambda_{2}})^{2}||K^{\ast}f||^{2}.
			\end{equation*}
			By theorem \ref{8}, we conclude that $\{V_{j},\Gamma_{j},u_{j}\}_{j\in J}$ is a $(C,C^{'})-$controlled $K-g-$fusion frame for $H$.
		\end{proof}
		
		\bibliographystyle{amsplain}

	\end{document}